\documentclass[reqno,a4paper,12pt]{amsart} 

\usepackage{amsmath,amscd,amsfonts,amssymb}
\usepackage{mathrsfs,dsfont}

\usepackage{tikz}
\usetikzlibrary{patterns}
\usepackage{float}
\usepackage{appendix}
\usepackage{caption}
\usepackage{subcaption}
\usepackage{enumerate}

\numberwithin{equation}{section}
\numberwithin{figure}{section}

\def\R{\mathbb{R}}

\def\Z{\mathbb{Z}}

\def\1{\mathds{1}}

\def\dH{\dim_{\mathcal{H}}}

\renewcommand\leq{\leqslant}
\renewcommand\geq{\geqslant}

\renewcommand\hat{\widehat}

\newcommand{\supp}{\operatorname{supp}}

\newcommand{\diam}{\operatorname{diam}}
\newcommand{\dist}{\operatorname{dist}}

\theoremstyle{plain}
\newtheorem{thm}{Theorem}[section]
\newtheorem{lem}[thm]{Lemma}

\newtheorem{conj}[thm]{Conjecture}

\newtheorem*{claim*}{Claim}
\newtheorem*{thm*}{Theorem}

\theoremstyle{definition}

\newtheorem*{definition*}{Definition}
\newtheorem*{remarks*}{Remarks}
\newtheorem*{remark*}{Remark}

\newenvironment{enumerate-math}
{\begin{enumerate}
\addtolength{\itemsep}{5pt}
}
{\end{enumerate}}

\newenvironment{enumerate-text}
{\begin{enumerate}
\addtolength{\itemsep}{5pt}
}
{\end{enumerate}}

\begin{document}

\title{On Hausdorff dimension of radial projections}

\author{Bochen Liu}
\address{Department of Mathematics, the Chinese University of Hong Kong, Shatin, N.T., Hong Kong}
\email{Bochen.Liu1989@gmail.com}

\date{}


\begin{abstract}
For any $x\in\mathbb{R}^d$, $d\geq 2$, denote $\pi^x: \mathbb{R}^d\backslash\{x\}\rightarrow S^{d-1}$ as the radial projection
$$\pi^x(y)=\frac{y-x}{|y-x|}. $$

Given a Borel set $E\subset\R^d$, $\dim_{\mathcal{H}} E\leq d-1$, in this paper we investigate for how many $x\in \mathbb{R}^d$ the radial projection $\pi^x$ preserves the Hausdorff dimension of $E$, namely whether $\dim_{\mathcal{H}}\pi^x(E)=\dim_{\mathcal{H}} E$. We develop a general framework to link $\pi^x(E)$, $x\in F$ and $\pi^y(F)$, $y\in E$, for any Borel set $F\subset\mathbb{R}^d$. In particular, whether $\dim_{\mathcal{H}}\pi^x(E)=\dim_{\mathcal{H}}E$ for some $x\in F$ can be reduced to whether $F$ is visible from some $y\in E$ (i.e. $\mathcal{H}^{d-1}(\pi^y(F))>0$). This allows us to apply Orponen's estimate on visibility to obtain 
$$\dim_{\mathcal{H}}\left\{x\in\mathbb{R}^d: \dim_{\mathcal{H}}\pi^x(E)<\dim_{\mathcal{H}}E\right\}\leq 2(d-1)-\dim_{\mathcal{H}}E,$$
for any Borel set $E\subset\R^d$, $\dim_{\mathcal{H}} E\in(d-2, d-1]$. 
This improves the Peres-Schlag bound when $\dim_{\mathcal{H}} E\in(d-\frac{3}{2}, d-1]$, and it is optimal at the endpoint $\dim_{\mathcal{H}} E=d-1$.
\end{abstract}
\maketitle

\section{Introduction}
For any $x\in\mathbb{R}^d$, $d\geq 2$, denote $\pi^x: \mathbb{R}^d\backslash\{x\}\rightarrow S^{d-1}$ as the radial projection
$$\pi^x(y)=\frac{y-x}{|y-x|}. $$

We say a Borel set $E\subset\R^d$ is 
\begin{itemize}
   \item visible from $x$, if $\mathcal{H}^{d-1}(\pi^x(E))>0$, and
	\item invisible from $x$, if $\mathcal{H}^{d-1}(\pi^x(E))=0$,
\end{itemize}
where $\mathcal{H}^{d-1}$ denotes the $(d-1)$-dimensional Hausdorff measure.

The study of visibility has a long history in geometry measure theory (see, for example, Section 6 in Mattila's survey \cite{Mat04}). It dates back to Marstrand's celebrated 1954 paper \cite{Mar54}, where it is proved that given a Borel set $E\subset\R^2$, $0<\mathcal{H}^s(E)<\infty$ for some $s>1$, then $E$ is visible from almost all $x\in\R^2$, and also from $\mathcal{H}^s$ almost all $x\in E$. Recently, due to Mattila-Orponen \cite{MO16} and Orponen \cite{Orp18} \cite{Orp19}, it is proved that given any Borel set $E\subset\R^d$, $\dH E>d-1$, then
\begin{equation}\label{sharp-radial-projection}\dim_{\mathcal{H}}\left\{x\in\mathbb{R}^d: \mathcal{H}^{d-1}\left(\pi^x(E)\right)=0\right\}\leq 2(d-1)-\dH E. \end{equation}
This dimensional exponent is sharp \cite{Orp18}: for any $\alpha\in (d-1, d]$, there exists a Borel set $E\subset\R^d$, $\dH E=\alpha$, which is invisible from a set of Hausdorff dimension $2(d-1)-\alpha$. 

When $0<\mathcal{H}^{d-1}(E)<\infty$, visibility depends on rectifiability. One can see, for example, Section 6 in Mattila's survey \cite{Mat04}, more recent results in \cite{SS06}, \cite{OS11}, \cite{BLZ16}, and references therein.

For any $\omega\in S^{d-1}$, denote $P_\omega: \R^d\rightarrow \omega^\perp$ as the orthogonal projection. The well-known Marstrand projection theorem states that: given a Borel set $E\subset\R^2$, then for almost all $\omega\in S^1$,
\begin{itemize}
 	\item $|P_\omega (E)|>0$, if $\dH E>1$;
 	\item $\dH P_\omega (E)=\dH E$, if $\dH E\leq 1$.
 \end{itemize} 
Inspired by Marstrand's result, for any projection $\pi$ from a higher dimensional space $U$ to a lower dimensional space $V$, the following two questions are natural.
\begin{enumerate}
	\item[Q1:] Given $E\subset U$, $\dim(E)>\dim(V)$, does $\pi(E)$ have positive density in $V$?
	\item[Q2:] Given $E\subset U$, $\dim(E)\leq \dim(V)$, does $\pi(E)$ preserve the dimension of $E$?
\end{enumerate}

For radial projections, Q1 has already been answered by \eqref{sharp-radial-projection}. So it remains to consider Q2, namely given $\dH E\leq d-1$ whether
$$\dH \pi^x(E)=\dH E.$$

There is very little known on Hausdorff dimension of radial projections. The Peres-Schlag machinery \cite{PS00} implies that when $\dH E\leq d-1$,
\begin{equation}\label{Peres-Schlag-1}\dim_{\mathcal{H}}\left\{x\in\mathbb{R}^d: \dim_{\mathcal{H}}\pi^x(E)<\tau\right\}\leq \tau+1,\ \forall \tau\in (\dH E-1, \dH E);\end{equation}
\begin{equation}\label{Peres-Schlag-2}\dim_{\mathcal{H}}\pi^x(E)\geq \dH E -1, \ \forall x\in\R^d.\end{equation}
In particular, \eqref{Peres-Schlag-1} implies
\begin{equation}\label{Peres-Schlag-3}\dim_{\mathcal{H}}\left\{x\in\mathbb{R}^d: \dim_{\mathcal{H}}\pi^x(E)<\dH E\right\}\leq \dH E+1.\end{equation}

Given $E\subset\R^2$ which does not lie in a line, Orponen \cite{Orp19} showed
\begin{equation}\label{Orponen-half}\dH\{x\in\R^2: \dH \pi^x(E)< \frac{\dH E}{2}\}=0.\end{equation}

None of these results seems to be sharp. The Peres-Schlag bound is trivial when $\tau=\dH E=d-1$, and Orponen's bound still has a gap to our expectation  $\dH \pi^x(E)=\dH E$. 

In this paper we introduce a new framework to study radial projections. We discover a connection between $\pi^x(E)$, $x\in F$ and $\pi^y(F)$, $y\in E$, for any Borel set $F\subset\R^d$. Roughly speaking whether $\dim_{\mathcal{H}}\pi^x(E)\geq\tau$ for some $x\in F$ can be reduced to whether $\dim_{\mathcal{H}}\pi^y(F)\geq d-1+\tau-\dH E$ for some $y\in E$. In particular, whether $\dim_{\mathcal{H}}\pi^x(E)=\dim_{\mathcal{H}} E$ for some $x\in F$ can be reduced to whether $\mathcal{H}^{d-1}(\pi^y(F))>0$ for some $y\in E$. This allows us to apply Orponen's estimate \cite{Orp19} in the proof of \eqref{sharp-radial-projection} to obtain the following.

\begin{thm}
	\label{main}
	Given a Borel set $E\subset\R^d$, $d\geq 2$, $\dH E\in(d-2, d-1]$, then
	$$\dim_{\mathcal{H}}\left\{x\in\mathbb{R}^d: \dim_{\mathcal{H}}\pi^x(E)<\dim_{\mathcal{H}}E\right\}\leq 2(d-1)-\dH E.$$
\end{thm}
Equivalently, given Borel sets $E, F\subset\R^d$, $\dH E\in (d-2, d-1]$, $\dH F>2(d-1)-\dH E$, then there exists $y\in F$ such that $\dim_{\mathcal{H}}\pi^x(E)=\dim_{\mathcal{H}} E$.

Theorem \ref{main} improves \eqref{Peres-Schlag-3} when $\dH E\in(d-\frac{3}{2}, d-1]$. It is optimal at the endpoint $\dH E=d-1$ (e.g. a hyperplane), but I don't think it is generally sharp. It would be  nice to show given $\dH E\leq d-1$,
$$\dim_{\mathcal{H}}\left\{x\in\mathbb{R}^d: \dim_{\mathcal{H}}\pi^x(E)<\dim_{\mathcal{H}}E\right\}\leq d-1,$$
or any counterexample. More generally I am wondering if the following holds. Notice it is always better than the Peres-Schlag bound \eqref{Peres-Schlag-3}.
\begin{conj}
	\label{conj}
	Given a Borel set $E\subset\R^d$, $d\geq 2$, $\dH E\in(k-1, k]$, $k=1, 2, \dots, d-1$, then
	$$\dim_{\mathcal{H}}\left\{x\in\mathbb{R}^d: \dim_{\mathcal{H}}\pi^x(E)<\dim_{\mathcal{H}}E\right\}\leq k.$$
\end{conj}

One may also wonder if anything new follows by plugging \eqref{Peres-Schlag-1}, \eqref{Peres-Schlag-2}, \eqref{Orponen-half} into our framework. Unfortunately, the answer is negative. However, we still expect that our method would help for improvement in the future. See Section \ref{discussion} for discussion.


\subsection{Radial projection and the distance problem}
Some ideas in this paper are inspired by recent work on the Falconer distance conjecture, which states that when the Hausdorff dimension of $E\subset\R^d$, $d\geq 2$ is large enough, the set of distances
$$\Delta(E)=\{|y-x|: x, y\in E\}$$
must have positive Lebesgue measure. In a recent paper of Guth, Iosevich, Ou and Wang \cite{GIOW18}, they group wave packets into good/bad families and work on them separately. Although in this paper we don't work with wave packets, the idea of good/bad tubes still helps. Their estimate also works on the Hausdorff dimension of distance sets (see \cite{Liu18-dimension}, \cite{shm17}).

On the other hand, in recent work on distances \cite{KS18} \cite{GIOW18}, a radial projection estimate due to Orponen \cite{Orp19} plays an important role. Therefore it has its own interest to see that these two problems interact with each other.

\vskip.125in

{\bf Notation.} $X\lesssim Y$ means $X\leq CY$ for some constant $C>0$. $X\approx Y$ means $X\lesssim Y$ and $Y\lesssim X$.


Throughout this paper $\phi\in\ C_0^\infty(\R^d)$, $\phi\geq 0$, $\int \phi =1$, and $\phi\geq 1$ on $B(0,\frac{1}{2})$. Denote $\phi_\delta(\cdot)=\frac{1}{\delta^d}\phi(\frac{\cdot}{\delta})$ and $\mu^{\delta}=\mu*\phi_\delta$ for any compactly supported Borel measure $\mu$ on $\R^d$.


$\widehat{f}(\xi):=\int e^{-2\pi i x\cdot\xi} f(x)\,dx$ denotes the Fourier transform.

For any finite Borel measure $\mu$ on $\R^d$, denote $||\mu||=\mu(\R^d)$.

\vskip.125in
{\bf Acknowledgments.} I would like to thank Tuomas Orponen for comments, and reminding me of adding Conjecture \ref{conj} into literature.

\section{Preliminaries}
\subsection{Frostman measures}
Suppose $E\subset\R^d$ is a Borel set. It is well known that for any $s_E<\dH E$ there exists a finite Borel measure $\mu_E$ supported on $E$, called a Frostman measure, such that
\begin{equation}\label{ball-condition-E}\mu_E(B(x,r))\leq r^{s_E},\ \forall\ x\in\R^d,\  r>0. \end{equation}
As a consequence, for any $s<s_E$ the energy integral
\begin{equation}\label{energy-integral-F}I_{s}(\mu_E)=\iint |x-y|^{-s}\,d\mu_E(x)\,d\mu_E(y)=c_d\int |\widehat{\mu_E}(\xi)|^2\,|\xi|^{-d+s}\,d\xi<\infty.\end{equation}
For more information about Frostman measures and energy integrals, see, for example, \cite{Mat15}, Section 2.5, 3.5.

\subsection{Hausdorff dimension of radial projections}
We shall use the following criteria to determine Hausdorff dimension. It is standard in geometric measure theory. Throughout this paper $\delta_k=2^{-k}$. 
\begin{lem}\label{criteria} 
Given $E\subset\R^d$, $x\in \R^d$, and a finite Borel measure $\mu_E$ on $E$. Suppose there exist $\tau\in(0, d-1]$, $K\in\Z_+$, $\beta>0$ such that
	$$\mu_E(\{y: \pi^x(y) \in D_k\})<\delta_k^{\beta}$$ 
	for any
	$$D_k\in\Theta_k^\tau:=\left\{\bigcup_{j=1}^M \theta_j:  M\leq \delta_k^{-\tau},\, \theta_j\subset S^{d-1} \text{ is a $\delta_k$-cap}, \,j=1,\dots,M\right\}.$$Then
	$$\dH \pi^x(E)\geq\tau. $$
\end{lem}
We give the proof for completeness.
\begin{proof}
If
$$\dH \pi^x(E)<\tau, $$
there exists $s\in(\dH \pi^x(E), \tau)$ such that $\mathcal{H}^{s}(\pi^x(E))=0$. By the definition of Hausdorff measure there exists an integer $N_0>0$ such that for any integer $N>N_0$, we can find a cover $\Theta$ of $\pi^x(E)$ consisting of finitely many caps of radius $<2^{-N}$, such that
$$\sum_{\theta\in\Theta} \diam(\theta)^s\leq 1. $$

Denote $$\Theta_k=\{\theta\in\Theta: 2^{-k-1}\leq\diam(\theta)< 2^{-k}\},\  k=N, N+1,\dots$$ and $$D_k=\bigcup_{\theta\in \Theta_k} \theta.$$
We may assume $N>\frac{s}{\tau-s}$. Then
$$\#(\Theta_k)\leq 2^{(k+1)s}\leq 2^{k\tau}.$$

Since 
$$\pi^x(E)\subset \bigcup_{k\geq N} D_k,$$and$$\mu_E\left(\left\{y\in E: \pi^x(y)\in \pi^x(E)\right\}\right)=||\mu_E||,$$ it follows that there exists $k_0\geq 0$ such that
$$\mu_E(\{y\in E: \pi^x(y)\in D_{k_0+N}\})>\frac{||\mu_E||}{100 (k_0+1)^2}. $$

On the other hand, for any $N>K$, the assumption in the lemma implies
$$\mu_E(\{y\in E: \pi^x(y)\in D_{k_0+N}\})<2^{-(N+k_0)\beta},$$
which is a contradiction when $N$ is large enough so that $$2^{-N\beta}<\inf_{k\geq 0}\frac{2^{k\beta}}{100(k+1)^2}||\mu_E||.$$
\end{proof}

\section{From $\pi^x(E)$, $x\in F$ to $\pi^y(F)$, $y\in E$}\label{framework}
Given Borel sets $E, F\subset\R^d$ and $0\leq\tau\leq\dH E$, in this section we investigate if there exists $x\in F$ such that $\dH \pi^x(E)\geq \tau$. Without loss of generality we may assume both $E, F$ are compact, $dist(E, F)\approx 1$, $\dH E, \dH F>0$ and $0<\tau<\dH E$. Let $\mu_E, \mu_F$ be Frostman measures on $E, F$ satisfying \eqref{ball-condition-E} with $\tau<s_E<\dH E$, $s_F<\dH F$.

To show $\dH \pi^x(E)\geq \tau$ for some $x\in F$, it suffices to show the condition in Lemma \ref{criteria} is satisfied for $\mu_F$-a.e. $x\in F$.

Let $\beta>0$ be a small number that will be specified later. Denote 
$F$ as the subset of $F$ consisting of points $x\in F$ such that
$$\mu_E(\{y:\pi^x(y)\in D_k\})\geq \delta_k^{\beta}$$
for some $D_k\in \Theta_k^\tau$. In other words, $F$ consists of points $x\in F$ where the condition in Lemma \ref{criteria} fails at scale $\delta_k$. Then
\begin{equation}\label{initial}\begin{aligned}\delta_k^\beta\mu_{F}(F_k)\leq &\int_{F_k} \sup_{D_k\in\Theta_k^\tau}\mu_E\left(\{y:\pi^x(y)\in D_k\}\right)\,d\mu_{F}(x)\\\leq &\int \sup_{D_k\in\Theta_k^\tau}\mu_E\left(\{y:\pi^x(y)\in D_k\}\right)\,d\mu_{F}(x). \end{aligned}\end{equation}

If one could show that
\begin{equation}
	\label{goal}
	\int \sup_{D_k\in\Theta_k^\tau}\mu_E\left(\{y:\pi^x(y)\in D_k\}\right)\,d\mu_{F}(x)\lesssim \delta_k^{2\beta},
\end{equation}
then $\mu_{F}(F_k)$ is summable in $k$. By the Borel-Cantelli Lemma, the condition in Lemma \ref{criteria} is satisfied for $\mu_F$-a.e. $x\in F$, as desired.

It remains to show \eqref{goal}. For any $x, y\in\R^d$, $x\neq y$, denote $l_{x,y}$ as the line passing through $x,y$ and $T_{x,y}^{k}$ as the $(10\delta_k)$-neighborhood of $l_{x,y}$. Take $s=d-1+\tau-s_E+4\beta$, define
$$Good_{k,s}=\left\{(x,y)\in\R^d\times\R^d: \mu_{F}(T_{x,y}^k)<\delta_k^s\right\},\ Good_{k,s}^x=\{y:(x,y)\in Good_{k,s}\};$$ $$Bad_{k,s}=\left\{(x,y)\in\R^d\times\R^d: \mu_{F}(T_{x,y}^k)\geq \delta_k^s\right\},\ Bad_{k,s}^x=\{y:(x,y)\in Bad_{k,s}\}, $$
and denote
$$\mu_{E, Good_{k,s}^x}=\mu_E\big|_{Good_{k,s}^x}, \ \mu_{E, Bad_{k,s}^x}=\mu_E\big|_{Bad_{k,s}^x}. $$

Under this notation, for each $x\in F$, we have $\mu_E=\mu_{E, Good_{k,s}^x}+\mu_{E, Bad_{k,s}^x}$, where $\mu_{E, Good_{k,s}^x}$, $\mu_{E, Bad_{k,s}^x}$ are both compactly supported Borel measures. Then the left hand side of \eqref{goal} is bounded above by
\begin{equation}\label{Good-Bad}
\begin{aligned}&\int \sup_{D_k\in\Theta_k^\tau}\mu_{E, Good_{k,s}^x}\left(\{y:\pi^x(y)\in D_k\}\right)\,d\mu_{F}(x)+ \int\mu_{E, Bad_{k,s}^x}(\R^d)\,\mu_{F}(x)\\=&\, \bf{Good} + \bf{Bad}.
\end{aligned}
\end{equation}

To move on, we need the following observation. For any $D_k\in\Theta_k^\tau$, denote $\widetilde{D_k}$ as the $\delta_k$-neighborhood of $D_k$.
\begin{lem}\label{observation} For any Borel measure $\mu$ on $\R^d$, $D_k\in\Theta_k^\tau$, $x\in\R^d$, $dist(x, \supp(\mu))\gtrsim 1$, we have
	$$\mu\left(\{y: \pi^x(y)\in D_k\}\right) \lesssim \int_{\widetilde{D_k}}\pi^x_*(\mu^{\delta_k})(\omega)\,d\omega.$$
Here the implicit constant only depends on $dist(x, \supp(\mu))$.
\end{lem}
\begin{proof}
Notice the right hand side equals
\begin{equation*}\begin{aligned} \int_{\pi^x(z)\in\widetilde{D_k}} \mu^{\delta_k}(z)\,dz= &\, \delta_k^{-d}\iint_{\pi^x(z)\in\widetilde{D_k}} \phi(\delta_k^{-1}(z-y))\,d\mu(y)\,dz\\\geq&\, \delta_k^{-d}\iint_{\pi^x(z)\in\widetilde{D_k},\  |y-z|\leq \frac{\delta_k}{2}}\,dz\,d\mu(y).\end{aligned}\end{equation*}

Fix $y$ and integrate $z$ first. Since $|x-y|\gtrsim 1$, we have $\pi^x(B(y, \frac{\delta_k}{2}))\subset \widetilde{D_k}$ if $\pi^x(y)\in D_k$. Therefore this integral is
$$\gtrsim \int_{\pi^x(y)\in D_k}\,d\mu(y)=\mu(\{y: \pi^x(y)\in D_k\}). $$
\end{proof}

\subsection{Estimate of \bf{Good}}
With $\mu=\mu_{E, Good_{k,s}^x}$ in Lemma \ref{observation}, it follows that
\begin{equation}\label{main-term}
{\bf Good} \lesssim \int \left(\sup_{D_k\in\Theta_k^\tau}\int_{\widetilde{D_k}}\pi^x_*(\mu_{E, Good_{k,s}^x}^{\delta_k})(\omega)\,d\omega\right)d\mu_{F}(x).\end{equation}

Notice for each $D_k\in \Theta_k^\tau$, $\widetilde{D}_k$ can be covered by $\lesssim\delta_k^{-\tau}$ caps of radius $\delta_k$. Therefore by Cauchy-Schwartz, {\bf Good} is bounded above by
\begin{equation}\label{reduce-main-term}
\begin{aligned}& \int \left(\sup_{D_k\in\Theta_k^\tau}|\widetilde{D_k}|\int_{\widetilde{D_k}}|\pi^x_*(\mu_{E,Good_{k,s}^x}^{\delta_k})(\omega)|^2\,d\omega\right)^\frac{1}{2}d\mu_{F}(x)\\\leq & \sup_{D_k\in\Theta_k^\tau}|\widetilde{D_k}|^\frac{1}{2}\int \left(\int_{S^{d-1}}|\pi^x_*(\mu_{E,Good_{k,s}^x}^{\delta_k})(\omega)|^2\,d\omega\right)^\frac{1}{2}d\mu_{F}(x)\\\lesssim &\,\delta_k^{\frac{d-1-\tau}{2}}\left(\iint_{S^{d-1}} |\pi^x_*(\mu_{E,Good_{k,s}^x}^{\delta_k})(\omega)|^2\,d\omega\,d\mu_{F}(x)\right)^\frac{1}{2}.
\end{aligned}
\end{equation}

Suppose $f\in C_0(\R^d)$ (the space of compactly supported continuous functions on $\R^d$), $\nu$ is a compactly supported Borel measure on $\R^d$ and $\supp(f)\cap \supp(\nu)=\emptyset$. Orponen's identity (see Lemma 3.1, \cite{Orp19}) implies that for any $p\in(0,\infty)$,
\begin{equation}\label{Orponen-id}\int \left|\left|\pi^x_* f\right|\right|_{L^p(S^{d-1})}^p\,d\nu(x) \approx \int_{S^{d-1}}||(P_\omega)_*f||_{L^p((P_\omega)_*\nu)}^p\,d\omega. \end{equation}
Unfortunately, since our function $\mu_{E,Good_{k,s}^x}^{\delta_k}$ depends on $x$, one cannot apply \eqref{Orponen-id} to \eqref{reduce-main-term} directly. Despite this, the idea of reducing radial projections to orthogonal projections still works. We will work on \eqref{reduce-main-term} very carefully. 

Since $\mu_{E,Good_{k,s}^x}^{\delta_k}\in C_0^\infty$, $dist(E, F)\approx 1$, we have for each $\omega\in S^{d-1}$, $x\in F$,
$$\pi^x_*(\mu_{E,Good_{k,s}^x}^{\delta_k})(\omega)\approx\int_{t\approx 1} \mu_{E,Good_{k,s}^x}^{\delta_k}(x+t\omega)\,dt. $$

Fix $\omega\in S^{d-1}$ and consider
\begin{equation}\label{fix-omega}\int\left|\pi^x_*(\mu_{E,Good_{k,s}^x}^{\delta_k})(\omega)\right|^2\,d\mu_F(x)\approx\int \left|\int_{t\approx 1} \mu_{E,Good_{k,s}^x}^{\delta_k}(x+t\omega)\,dt\right|^2\,d\mu_{F}(x).\end{equation}

We claim that one can replace $\mu_{F}$ by a compactly supported Borel measure $\mu_{F,\, \omega}$, without changing the value of \eqref{fix-omega}, such that
\begin{equation}\label{Frostman-proj}(P_\omega)_*\mu_{F, \omega}(B(u, \delta_k))\lesssim \delta_k^{s},\ \forall\ u\in\R^{d-1}.\end{equation}

To see this, notice $\mu_{E,Good_{k,s}^x}^{\delta_k}$ is supported on the $\delta_k$-neighborhood of $Good_{k,s}^x$, so
$$\int_{t\approx 1} \mu_{E,Good_{k,s}^x}^{\delta_k}(x+t\omega)\,dt $$
is nontrivial only if there exists $y\in Good_{k,s}^x$ such that the distance between $y$ and the line $\{x+t\omega: t\in \R\}$ is no more than $\delta_k$, which is equivalent to $dist(P_\omega y, P_\omega x)\leq \delta_k$. This means, for each fixed $\omega\in S^{d-1}$, we can restrict $\mu_{F}$ to
$$\left\{x\in F: dist(P_\omega (Good_{k,s}^x), P_\omega x)\leq \delta_k  \right\} $$
without changing the value of \eqref{fix-omega}. Denote this restricted measure as $\mu_{F,\, \omega}$. 

Now let us check \eqref{Frostman-proj}. By our construction of $\mu_{F,\, \omega}$, given any $\delta_k$-tube $T^k$ of direction $\omega$, we have $\mu_{F,\,\omega}(T^k)\neq 0$ only if $T^k\cap F$ is contained in the $\delta_k$-neighborhood of a line $l_{x, y}$, $(x, y) \in Good_{k,s}$, of direction $\omega$. This implies $T^k\subset T^k_{x,y}$ and therefore
$$\mu_{F, \omega}(T^k)\leq \mu_{F}(T^k_{x,y})\leq \delta_k^s,$$
as desired.

Now, with $\mu_{F}$ replaced by $\mu_{F_{k, \omega}}$, \eqref{fix-omega} is reduced to
\begin{equation}\begin{aligned}\label{to-orthogonal-proj}\int \left|\int_{t\approx 1} \mu_{E,Good_{k,s}^x}^{\delta_k}(x+t\omega)\,dt\right|^2\,d\mu_{F, \,\omega}(x)\leq & \int \left|\int_{t\approx 1} \mu_{E}^{\delta_k}(x+t\omega)\,dt\right|^2\,d\mu_{F, \,\omega}(x)\\\approx & \int \left|(P_\omega)_* \mu_{E}^{\delta_k}(P_\omega x)\right|^2\,d\mu_{F, \,\omega}(x) \\= & \int \left|(P_\omega)_* \mu_{E}^{\delta_k}(u)\right|^2\,d (P_\omega)_*\mu_{F, \,\omega}(u).
\end{aligned}
\end{equation}

We shall show, up to a negligible error term,
\begin{equation}\label{remove-F}\int \left|(P_\omega)_* \mu_{E}^{\delta_k}(u)\right|^2\,d (P_\omega)_*\mu_{F, \,\omega}(u)\lesssim \delta_k^{-d+1+s}\int_{\xi\in \R^{d-1}, |\xi|\lesssim \delta_k^{-1}} \left|\widehat{(P_\omega)_* \mu_{E}}(\xi)\right|^2\, d\xi.\end{equation}

Here the negligible error comes from the fact that $\widehat{\mu_{E}^{\delta_k}}(\xi)=\widehat{\mu_{E}}(\xi) \hat{\phi}(\delta_k\xi)$ is essentially supported on $B(0, \delta_k^{-1})$. This allows us to replace $|\widehat{\mu_{E}^{\delta_k}}(\xi)|$ by $|\widehat{\mu_{E}}(\xi)|\chi_{B(0, \delta_k^{-1})}$ up to a negligible error.

For any $f\in L^2\left((P_\omega)_*\mu_{F, \,\omega}\right)$, by Fourier inversion we have, up to a negligible error,
\begin{equation}\begin{aligned}&\left(\int (P_\omega)_* \mu_{E}^{\delta_k}(u)\cdot f(u)\,d (P_\omega)_*\mu_{F, \,\omega}(u)\right)^2\\ \lesssim &\left(\int_{\xi\in \R^{d-1}, |\xi|\lesssim \delta_k^{-1}} \left|\widehat{(P_\omega)_* \mu_{E}}(\xi)\right|\, \left|\widehat{f\,d(P_\omega)_*\mu_{F, \,\omega}}(\xi)\right|\,d\xi\right)^2\\\leq&\int_{\xi\in \R^{d-1}, |\xi|\lesssim \delta_k^{-1}} \left|\widehat{(P_\omega)_* \mu_{E}}(\xi)\right|^2\, d\xi\cdot \int_{\xi\in \R^{d-1}, |\xi|\lesssim \delta_k^{-1}}\left|\widehat{f\,d(P_\omega)_*\mu_{F, \,\omega}}(\xi)\right|^2\,d\xi.
\end{aligned}
\end{equation}

Then \eqref{remove-F} follows from the following lemma and \eqref{Frostman-proj}, with $r=\delta_k$, $m=d-1$.
\begin{lem}\label{Shur-test}
	Given $r>0$, fixed. Let $\mu$ be a Borel measure on $\R^m$ such that $\mu(B(u,r))\lesssim r^s$ for any $x\in\R^m$. Then 
	$$\int_{|\xi|\lesssim r^{-1}}|\widehat{f\,d\mu}(\xi)|^2\,d\xi\lesssim r^{-m+s}||f||_{L^2(\mu)}^2. $$
\end{lem}
\begin{proof}
	The proof is standard. Let $\psi$ be a positive Schwartz function on $\R^m$ whose Fourier transform has bounded support. Then 
	$$\int_{|\xi|\lesssim r^{-1}}|\widehat{f\,d\mu}(\xi)|^2\,d\xi\lesssim\int|\widehat{f\,d\mu}(\xi)|^2\psi(r\xi)\,d\xi = r^{-m} \int f(x) K(x,y) f(y)\,d\mu(x)\,d\mu(y), $$
	where $K(x,y)=\hat{\psi}(\frac{x-y}{r})$. Since
	$$\int |K(x,y)|\,d\mu(x)\lesssim \int_{|x-y|\lesssim r}\,d\mu(x)\lesssim r^s, $$
	$$\int |K(x,y)|\,d\mu(y)\lesssim \int_{|x-y|\lesssim r}\,d\mu(y)\lesssim r^s,$$
	the lemma follows by Shur's test.
\end{proof}

Due to \eqref{to-orthogonal-proj}, \eqref{remove-F}, we have the following estimate of \eqref{fix-omega}: for each fixed $\omega\in S^{d-1}$,
$$\int \left|\int_{t\approx 1} \mu_{E,Good_{k,s}^x}^{\delta_k}(x+t\omega)\,dt\right|^2\,d\mu_{F}(x)\lesssim \delta_k^{-d+1+s}\int_{\xi\in \R^{d-1}, |\xi|\lesssim \delta_k^{-1}} \left|\widehat{(P_\omega)_* \mu_{E}}(\xi)\right|^2\, d\xi.$$

Plug this into \eqref{reduce-main-term}. It follows that
\begin{equation}\label{main-term-estimate}\begin{aligned}{\bf Good} & \lesssim \delta_k^{\frac{d-1-\tau}{2}} \cdot \delta_k^{\frac{-d+1+s}{2}}\cdot \left(\int_{S^{d-1}}\int_{\xi\in\R^{d-1}, |\xi|\lesssim \delta_k^{-1}} \left|\widehat{(P_\omega)_* \mu_{E}}(\xi)\right|^2\, d\xi\,d\omega\right)^\frac{1}{2}\\ &= c_d\,\delta_k^{\frac{s-\tau}{2}}\cdot \left(\int_{\eta\in\R^{d}, |\eta|\lesssim \delta_k^{-1}} \left|\widehat{\mu_{E}}(\eta)\right|^2\, |\eta|^{-1}\, d\eta\right)^\frac{1}{2}\\&\lesssim \delta_k^{\frac{s-\tau}{2}}\cdot \left(\sum_{j\leq k} \delta_{j}\,\int_{\eta\in\R^{d}, |\eta|\lesssim \delta_j^{-1}} \left|\widehat{\mu_{E}}(\eta)\right|^2\, d\eta\right)^\frac{1}{2}\\&\lesssim \delta_k^{\frac{-d+1+s_E+s-\tau}{2}}  \\&=\delta_k^{2\beta} ,\end{aligned}\end{equation}
where the second to the last line follows from Lemma \ref{Shur-test} and our assumption $s_E< d-1$.

\subsection{Estimate of \bf{Bad}}
Now it remains to show
$${\bf Bad}\lesssim \delta_k^{2\beta}.$$

By the symmetry of $x, y$ in the definition of $Bad_{k,s}$, we have
$${\bf Bad} =  \int\mu_{E, Bad_{k,s}^x}(\R^d)\,\mu_{F}(x)=\mu_E\times\mu_{F}(Bad_{k,s})=\int \mu_{F}(Bad_{k,s}^y)\,d\mu_E(y).$$

Fix $y\in E$. Since $\mu_{F}(T_{x,y}^k)>\delta_k^s$ for any $x\in F\cap Bad_{k,s}^y$, by the $5r$-covering lemma $F\cap Bad_{k,s}^y$ can be covered by $\lesssim \delta_k^{-s}$ many $\delta_k$-tubes whose central lines pass through $y$. Therefore
\begin{equation}\label{dual-form}{\bf Bad}\leq \int  \sup_{D_k\in\Theta_k^s}\mu_{F}\left(\{x:\pi^y(x)\in D_k\}\right)\,d\mu_E(y)\end{equation}
and the whole problem, namely whether $\pi^x(E)\geq\tau$ for some $x\in F$, is reduced to
\begin{equation}\label{desired-estimate-bad}\int  \sup_{D_k\in\Theta_k^s}\mu_{F}\left(\{x:\pi^y(x)\in D_k\}\right)\,d\mu_E(y)\lesssim\delta_k^{2\beta}.\end{equation}

Notice that \eqref{desired-estimate-bad} has the same type of \eqref{goal}, with $E, F$ swapped and $s$ in place of $\tau$. Therefore \eqref{desired-estimate-bad} itself would imply $\dH \pi^y(F)\geq s$ for some $y\in E$. In other words, we start from $\dH \pi^x(E)\geq\tau$, $x\in F$, and eventually end up with $\dH \pi^y(F)\geq s=d-1+\tau-\dH E$, $y\in E$. This allows us to apply known estimates on $\pi^y(F)$, $y\in E$ to obtain improvement on $\pi^x(E)$, $x\in F$.

\section{Proof of Theorem \ref{main}}
It suffices to show \eqref{desired-estimate-bad} for any $\tau<s_E<\dH E$ when $s_F>2(d-1)-s_E$. By Lemma \ref{observation} and H\"older, the left hand side of \eqref{desired-estimate-bad} is bounded above by \begin{equation}\begin{aligned}&\int \sup_{D_k\in\Theta_k^s}\left(\int_{\widetilde{D_k}}\pi^x_*(\mu_{F} ^{\delta_k})(\omega)\,d\omega\right)d\mu_E(x) \\\leq &\sup_{D_k\in\Theta_k^s}|\widetilde{D_k}|^\frac{1}{p'} \int \left(\int_{S^{d-1}}|\pi^x_*(\mu_{F} ^{\delta_k})(\omega)|^p\,d\omega\right)^\frac{1}{p}d\mu_E(x)\\\lesssim &\,\delta_k^\frac{d-1-s}{p'} \left(\iint_{S^{d-1}}|\pi^x_*(\mu_{F} ^{\delta_k})(\omega)|^p\,d\omega\, d\mu_E(x)\right)^\frac{1}{p}.\end{aligned}\end{equation}

Since $s_F>d-1$ and $s_E+s_F>2(d-1)$, Orponen's estimate (see (3.6) in \cite{Orp19}) in the proof of \eqref{sharp-radial-projection} implies that when $p>1$ is small enough (independent in $\delta_k$),
$$\iint_{S^{d-1}}|\pi^x_*(\mu_{F} ^{\delta_k})(\omega)|^p\,d\omega \,d\mu_E(x)\lesssim I_{s_E-\epsilon}(\mu_E)^\frac{1}{2p}\cdot I_{s_F-\epsilon}(\mu_{F}^{\delta_k})^\frac{1}{2}\leq I_{s_E-\epsilon}(\mu_E)^\frac{1}{2p}\cdot I_{s_F-\epsilon}(\mu_{F})^\frac{1}{2}, $$
where $\epsilon>0$ is a fixed small constant. Here the last inequality holds due to the Fourier expression of the energy integral \eqref{energy-integral-F}. 

Since $I_{s_E-\epsilon}(\mu_E), I_{s_F-\epsilon}(\mu_{F})<\infty$ and $s=d-1+\tau-s_E+4\beta$, for any $\tau<s_E$ we have
$${\bf Bad}\lesssim\delta_k^{2\beta}$$
when $\beta>0$ is small enough. This completes the proof of Theorem \ref{main}.

\section{Discussion}\label{discussion}
In this section we explain why no new result follows from \eqref{Peres-Schlag-1}, \eqref{Peres-Schlag-2}, \eqref{Orponen-half}. Recall $\dH E\leq d-1$ and $s=d-1+\tau-\dH E$.

By our framework, to consider if $\dim_{\mathcal{H}}\pi^x(E)\geq \tau$ for some $x\in F$, it suffices to consider if $\dim_{\mathcal{H}}\pi^y(F)\geq s$ for some $y\in E$. If we apply \eqref{Peres-Schlag-1}, \eqref{Peres-Schlag-2}, \eqref{Orponen-half}, the following three sufficient conditions are obtained:$$\dH E>s+1;$$ $$s< \dH F -1;$$ $$s< \frac{\dH F}{2},\ \text{if } d=2.$$

The first condition implies
$$\tau=s+\dH E-d+1<2\dH E - d \leq \dH E -1,$$
which is already trivial by \eqref{Peres-Schlag-2}.

The second condition implies
$$\dH F> s+1=d+\tau-\dH E \geq \tau+1$$
which does not beat \eqref{Peres-Schlag-1}.

The third condition implies
$$\dH F> 2+2\tau-2\dH E=2-\dH E +(2\tau-\dH E).$$
When $2\tau-\dH E\leq 0$, it does not beat \eqref{Orponen-half}; when $2\tau-\dH E\geq 0$, it does not beat Theorem \ref{main}. In fact, even if we could obtain $\frac{\dH E}{2}+\epsilon_0$ in \eqref{Orponen-half}, no improvement follows for the same reason.

Despite these negative examples, one may still expect our framework to help in the future. For instance, in the sharpness example of \eqref{sharp-radial-projection} (see \cite{Orp18}), the exception set lies in a hyperplane. Therefore it is reasonable to expect an improvement if we consider ``pins'' that do not lie in any hyperplane. If it were true, a corresponding improvement of Theorem \ref{main} would follow from our framework, under the extra assumption that $E$ does not lie in any hyperplane.

\bibliographystyle{abbrv}
\bibliography{/Users/MacPro/Dropbox/Academic/paper/mybibtex.bib}

\end{document}